\documentclass[12pt,reqno]{amsart}
\usepackage{fullpage,amssymb,amsmath,mathrsfs,enumerate}

\newtheorem{thm}{Theorem}[section]
\newtheorem{proposition}[thm]{Proposition}
\newtheorem{lemma}[thm]{Lemma}
\newtheorem{cor}[thm]{Corollary}
\newtheorem*{thm*}{Theorem}

\newtheoremstyle{named}{}{}{\itshape}{}{\bfseries}{.}{.5em}{\thmnote{#3}}
\theoremstyle{named}
\newtheorem*{namedtheorem}{Theorem}

\theoremstyle{definition}
\newtheorem*{defin}{Definition}
\newtheorem*{ack}{Acknowledgments}
\newtheorem{question}[thm]{Question}

\theoremstyle{remark}
\newtheorem{remark}[thm]{Remark}

\def\conv{\operatorname{conv}}
\def\er{\mathbb R}
\def\R{\mathcal R}
\def\O{\mathcal O}
\def\U{\mathcal U}

\def\en{\mathbb N}
\def\C{\mathcal{C}}
\def\ov{\overline}
\def\eps{\varepsilon}
\def \dens {\operatorname{dens}}

\begin{document}
\author{Marek C\'uth}
\title{Simultaneous projectional skeletons}
\thanks{The author was supported by the Research grant GA \v{C}R P201/12/0290}
\email{cuthm5am@karlin.mff.cuni.cz}
\address{Charles University, Faculty of Mathematics and Physics, Sokolovsk\'a 83, 186 75 Praha 8 Karl\'{\i}n, Czech Republic}
\subjclass[2010]{46B26,54D30}
\keywords{retractional skeleton, projectional skeleton, Valdivia compacta, Plichko spaces, Asplund spaces}
\begin{abstract}
We prove the existence of a simultaneous projectional skeleton for certain subspaces of $C(K)$ spaces. This generalizes a result on simultaneous projectional resolutions of identity proved by M. Valdivia. We collect some consequences of this result. In particular we give a new characterization of Asplund spaces using the notion of projectional skeleton.\end{abstract}
\maketitle
\section{Introduction}

Systems of bounded linear projections on Banach spaces are an important tool for the study of structure of nonseparable Banach spaces. They enable us to transfer properties from smaller (separable) spaces to larger ones.

One of the important concepts of such a system is a \textit{projectional resolution of the identity} (PRI, for short); see, e.g. \cite{hajek} and \cite{fabianCervenaKniha} for a definition and results on constructing a PRI in various classes of spaces.

However, even better knowledge of the Banach space provides a \textit{projectional skeleton}. The class of spaces with a projectional skeleton was introduced by W. Kubi\'s in \cite{kubis}. Spaces with a 1-projectional skeleton not only have a PRI, but they form a $\mathcal{P}$-class; see, e.g., \cite[Definition 3.45]{hajek} and \cite[Theorem 17.6]{kubisKniha}. Consequently, an inductive argument works well when ``putting smaller pieces from PRI together'' and we may prove those spaces inherit certain structure from separable spaces. For example, every space with a projectional skeleton has a strong Markushevich basis and an LUR renorming; see, e.g., {\cite[Theorem 5.1]{hajek}} and {\cite[Theorem VII.1.8]{DGZ}}. Moreover, it is possible to characterize some classes of other spaces (e.g. WLD, Plichko and Asplund spaces) in terms of a projectional skeleton; see \cite{kubis} for more details.

One of the largest class of spaces admitting a PRI is related to Valdivia compact spaces.
\begin{defin}
Let $\Gamma$ be a set. We put $\Sigma(\Gamma) = \{x\in\er^\Gamma:\;|\{\gamma\in\Gamma:\;x(\gamma)\neq 0\}|\leq\omega\}$. Given a compact $K$, $A\subset K$ is called a \textit{$\Sigma$-subset} of $K$ if there is a homeomorphic embedding $h:K\to[0,1]^\kappa$ such that $A = h^{-1}[\Sigma(\kappa)]$. A compact space $K$ is said to be \textit{Valdivia compact} if there exists a dense $\Sigma$-subset of $K$.
\end{defin}
The following result is contained in \cite{valdivia}. Let us just note that there is proved even something more in \cite{valdivia}, but we will be interested only in the following statement.
\begin{namedtheorem}[Theorem A]\text{\textnormal{({\cite[Theorem 1]{valdivia}})}}
Let $K$ be a Valdivia compact space with a dense $\Sigma$-subset $A$. Let $(Y_n)_{n\in\en}$ be a sequence of $\tau_p(A)$-closed subspaces of $\C(K)$. If $\dens \C(K) = \mu$, then there is a PRI $\{P_\alpha:\;\omega\leq\alpha\leq\mu\}$ in $\C(K)$ such that $P_\alpha(Y_n)\subset Y_n,\;n\in\en,\omega\leq\alpha\leq\mu$.
\end{namedtheorem}

The system of projections as above is called ``simultaneous projectional resolution of the identity'' in \cite{valdivia}. In the present paper we generalize Theorem A using the notion of a skeleton.

Let us have a partially ordered set $(\Gamma,<)$. We say that it is \textit{up-directed}, if for any $s,t\in\Gamma$, there is $u\in\Gamma$ such that $u\geq s$, $u\geq t$. We say that $\Gamma$ is \textit{$\sigma$-complete}, if for every increasing sequence $(s_n)_{n\in\en}$ in $\Gamma$, $\sup_{n\in\en}s_n$ exists.
\begin{defin}A \textit{projectional skeleton} in a Banach space $X$ is a family of projections $\{P_\gamma\}_{\gamma\in\Gamma}$, indexed by an up-directed $\sigma$-complete partially ordered set $\Gamma$, such that
\begin{enumerate}[\upshape (i)]
	\item Each $P_s X$ is separable.
	\item $X = \bigcup_{s\in\Gamma}P_s X$.
	\item $s\leq t \Rightarrow P_s = P_s\circ P_t = P_t\circ P_s.$
	\item Given $s_1 < s_2 < \cdots$ in $\Gamma$ and $t = \sup_{n\in\en}s_n$, $P_t X = \ov{\bigcup_{n\in\en}P_{s_n}X}$.
\end{enumerate}
Given $r\geq 1$, we say that $\{P_s\}_{s\in\Gamma}$ is an \textit{r-projectional skeleton} if it is a projectional skeleton such that $\|P_s\|\leq r$ for every $s\in\Gamma$.

We say that $\{P_s\}_{s\in\Gamma}$ is a \textit{commutative projectional skeleton} if $P_s\circ P_t = P_t\circ P_s$ for any $s,t\in\Gamma$.
\end{defin}

\begin{remark}\label{remarkSkeleton}
Having an $r$-projecitonal skeleton $\{P_s\}_{s\in\Gamma}$, an increasing sequence of indices $s_0 < s_1 < \cdots$ in $\Gamma$ and $t = \sup_{n\in\en}s_n$, it is easy to verify that $P_t(x) = \lim_nP_{s_n}(x)$ for every $x\in X$; see \cite[Lemma 10]{kubis}. This statement holds even for an arbitrary projectional skeleton, not neccessary uniformly bounded, but this will not be needed it any further. Recall that due to \cite{kubis}, we may always assume that every projectional skeleton is an r-projectional skeleton for some $r\geq 1$ (just by passing to a suitable cofinal subset of $\Gamma$).
\end{remark}

In \cite{kubisSmall} there was introduced a class of compact spaces with a retractional skeleton and it was observed in \cite{kubisSmall}, \cite{kubis} and \cite{cuthCMUC} that those spaces are more general than Valdivia compact spaces, but they share a lot of properties with them.

\begin{defin}A \textit{retractional skeleton} in a compact space $K$ is a family of retractions $\mathfrak{s} = \{r_s\}_{s\in\Gamma}$, indexed by an up-directed $\sigma$-complete partially ordered set $\Gamma$, such that
\begin{enumerate}[\upshape (i)]
	\item $r_s[K]$ is metrizable for each $s\in\Gamma$.
	\item For every $x\in K$, $x = \lim_{s\in\Gamma}r_s(x)$.
	\item $s\leq t \Rightarrow r_s = r_s\circ r_t = r_t\circ r_s.$
	\item Given $s_1 < s_2 < \cdots$ in $\Gamma$ and $t = \sup_{n\in\en}s_n$, $r_t(x) = \lim_{n\to\infty}r_{s_n}(x)$ for every $x\in K$.
\end{enumerate}
We say that $\{r_s\}_{s\in\Gamma}$ is a \textit{commutative retractional skeleton} if $r_s\circ r_t = r_t\circ r_s$ for any $s,t\in\Gamma$.\\
We say that $D(\mathfrak{s}) = \bigcup_{s\in\Gamma}r_s[K]$ is the \textit{set induced by a retractional skeleton} in $K$.
\end{defin}

By $\R_0$ we denote the class of all compacta which have a retractional skeleton.

The class of Banach spaces with a projectional skeleton (resp. class of compact spaces with a retractional skeleton) is closely related to the concept of Plichko spaces (resp. Valdivia compacta). By {\cite[Theorem 27]{kubis}}, Plichko spaces are exactly spaces with a commutative projectional skeleton. By {\cite[Theorem 6.1]{kubisSmall}}, Valdivia compact spaces are exactly compact spaces with a commutative retractional skeleton.

The above mentioned generalization of the result from \cite{valdivia} is the following.

\begin{thm}\label{tMain}Let us have  $K\in\R_0$ with $D\subset K$ induced by a retractional skeleton in $K$ and countably many $\tau_p(D)$-closed subspaces $(Y_n)_{n\in\en}$ of $\C(K)$. Then there exists a $1$-projectional skeleton $\{P_s\}_{s\in\Gamma}$ in $\C(K)$ such that $P_s(Y_n)\subset Y_n$, $n\in\en$, $s\in\Gamma$.

In particular, for all $n\in\en$, $\{P_s\upharpoonright_{Y_n}\}_{s\in\Gamma}$ is $1$-projectional skeleton in $Y_n$.
\end{thm}
The statement with PRI, instead of a projectional skeleton, follows immediately from the proof of \cite[Theorem 17.6]{kubisKniha}. Hence, this really is a generalization of Theorem A.

Moreover, we use the existence of a ``simultaneous skeleton'' to prove other statements concerning the structure of spaces with a projectional (resp. retractional) skeleton. We study a relationship between projectional and retractional skeletons. In particular, we give an answer to \cite[Question 1]{cuthCMUC}. We also study subspaces (resp. continuous images) of spaces with a projectional (resp. retractional) skeleton.

Using the above, we give the following characterization of Asplund spaces. A Banach space $X$ is Asplund if and only if the dual space has a 1-projectional skeleton after every renorming of $X$ if and only if the bidual unit ball has a retractional skeleton after every renorming of $X$. In particular, this gives an answer to {\cite[Question 1]{kalenda01}}. Let us just note that the answer has already been known to O. Kalenda before and it has been contained in one of his unpublished remarks.

The structure of the paper is as follows: first, we prove Theorem \ref{tMain}. Next, we use this result
to study the relationship between projectional and retractional skeletons. Then we characterize those subspaces (resp.
continuous images) of a space with a projectional (resp. retractional) skeleton, where a ``natural projectional
subskeleton'' exists. Next, we give a new characterization of Asplund spaces. Finally, we show some more applications of given
results.

\section{Preliminaries}

We denote by $\omega$ the set of all natural numbers (including $0$), by $\en$ the set $\omega\setminus\{0\}$.

All topological spaces are assumed to be Hausdorff. Let $T$ be a topological space. The closure of a set $A$ we denote by
$\ov{A}$. We say that $A\subset T$ is \textit{countably closed} if $\ov{C}\subset A$ for every countable $C\subset A$. A
topological space $T$ is a \textit{Fr\'echet-Urysohn space} if for every $A\subset T$ and every $x\in\ov{A}$ there is a sequence
$x_n\in A$ with $x_n\to x$. We say that $T$ is \textit{countably compact} if every countable open cover of $T$ has a finite
subcover. If $T$ is completely regular, we denote by $\beta T$ the Stone-\v{C}ech compactification of $T$.

Let $K$ be a compact space. By $\C(K)$ we denote the space of continuous
functions on $K$. Given a dense set $D\subset K$, we denote by $\tau_p(D)$ the topology of the pointwise convergence on $D$; i.e., the weakest topology on $C(K)$ such that $\C(K)\ni f\mapsto f(d)$ is continuous for every $d\in D$.
$P(K)$ stands for the space of probability measures with the $w^*$--topology (the $w^*$--topology is taken
from the representation of $P(K)$ as a compact subset of $(\C(K)^*,w^*)$).

We shall consider Banach spaces over the field of real numbers (but many results hold for complex spaces as well). If $X$ is a
Banach space and $A\subset X$, we denote by $\conv{A}$ the convex hull of $A$. We write $A^\bot = \{x^*\in X^*:\;(\forall a\in
A)\;x^*(a) = 0\}$. $B_X$ is the unit ball in $X$; i.e., the set
$\{x\in X:\; \|x\| \leq 1\}$. $X^*$ stands for the (continuous) dual space of $X$. For a set $A\subset X^*$ we denote by
$\ov{A}^{w^*}$ the $weak^*$ closure of $A$. Given a set $D\subset X^*$ we denote by $\sigma(X,D)$ the weakest topology on $X$
such that each functional from $D$ is continuous.

A set $D\subset X^*$ is \textit{r-norming} if
$$\|x\| \leq r. \sup\{|x^*(x)|:\;x^*\in D\cap B_{X^*}\}.$$
We say that a set $D\subset X^*$ is norming if it is $r$-norming for some $r\geq 1$.

Recall that a Banach space $X$ is called \textit{Plichko} (resp. 1-\textit{Plichko}) if there are a linearly dense set $M\subset X$ and a norming (resp. 1-norming) set $D\subset X^*$ such that for every $x^*\in D$ the set $\{m\in M:\;x^*(m)\neq 0\}$ is countable. 

\begin{defin}Let $\mathfrak{s} = \{P_s\}_{s\in\Gamma}$ be a projectional skeleton in a Banach space $X$ and let $D(\mathfrak{s}) = \bigcup_{s\in\Gamma}P^*_s[X^*]$. Then we say that \textit{$D(\mathfrak{s})$ is induced by a projectional skeleton}.\end{defin}

Some properties of a set induced by a retractional skeleton in $K\in\R_0$ are similar to the properties of a ``dense $\Sigma$-subset'' in a Valdivia compact $K$. Bellow we collect some of the most important statements. Those will be needed in what follows.

\begin{lemma}\label{lBasicSkeleton}Assume $D$ is induced by a retractional skeleton in $K$. Then:
\begin{enumerate}[\upshape (i)]
	\item $D$ is dense and countably closed in $K$.
	\item $K = \beta D$ and $D$ is a Fr\'echet-Urysohn space.
	\item If $F\subset K$ is closed and $F\cap D$ is dense in $F$, then $F\cap D$ is induced by a retractional skeleton in $F$.
	\item If $G\subset K$ is a $G_\delta$ set, then $G\cap D$ is dense in $G$. In particular, if $G\subset K$ is a closed $G_\delta$ set, then $G$ has a retractional skeleton.
	\item If $E\subset D$ is a countably closed and dense set in $K$, then $E = D$.
\end{enumerate}
\end{lemma}
\begin{proof}The statements (i) and (ii) are proved in \cite{kubis}, (iii) and (iv) are proved in \cite{cuthCMUC}. In order to prove (v), we follow the lines of \cite[Lemma 1.7]{kalendaSurvey}. Fix $x\in D$. Then $x\in \ov{E}$ and using the fact that $D$ is Fr\'echet-Urysohn, there exists a sequence $x_n\in E$ with $x_n\to x$. As $E$ is countably closed, $x\in E$.
\end{proof}

For other statements concerning similarities between Valdivia compacta and spaces with a retractional skeleton we refer to \cite{cuthCMUC} where more details may be found.

The last statement of this section is the following lemma which we will need later.

\begin{lemma}\label{lFrechetQuotient}Let $\varphi:D\to B$ be a continuous mapping, where $D$ is countably compact space and $B$ is Fr\'echet-Urysohn. Then $\varphi$ is closed.\end{lemma}
\begin{proof}Let $A\subset D$ be a closed set. Then $A$ is countably compact; hence, $\varphi(A)$ is countably compact. Fix $x\in\ov{\varphi(A)}$. As $B$ is Fr\'echet-Urysohn, there exists a sequence $(x_n)_{n=1}^\infty\subset \varphi(A)$ with $x_n\to x$. Since $\varphi(A)$ is countably compact, there exists a subnet $(x_\nu)$ of the sequence $(x_n)_{n=1}^\infty$ such that $x_\nu\to y\in \varphi(A)$. It follows that $y = x\in \varphi(A)$.
\end{proof}

\section{Simultaneous projectional skeletons}

In this section we prove Theorem \ref{tMain}.

Assume $D$ is induced by a retractional skeleton $\{r_s\}_{s\in\Gamma}$ in $K$. Let us define, for $s\in\Gamma$, the projection $P_s$ by $P_s(f) = f\circ r_s$, $f\in\C(K)$. It is known that $\{P_s\}_{s\in\Gamma}$ is a 1-projectional skeleton in $\C(K)$. Now, let us fix a set $\Gamma'\subset\Gamma$. We would like to know that $\{P_s\}_{s\in\Gamma'}$ is still a 1-projectional skeleton. It is easily seen that a sufficient condition for $\Gamma'$ is to be unbounded and $\sigma$-closed in $\Gamma$ in the sense of the following definition.

\begin{defin}Let $\Gamma$ be an up-directed $\sigma$-complete partially ordered set and $\Gamma'\subset\Gamma$. We say that $\Gamma'$ is
\begin{enumerate}[\upshape (i)]
 \item \textit{unbounded} (in $\Gamma$), if for every $s\in\Gamma$ there exists $t\in\Gamma'$ such that $s\leq t$;
 \item \textit{$\sigma$-closed} (in $\Gamma$), if for every increasing sequence $\{s_n\}_{n\in\en}$ in $\Gamma'$, $\sup s_n\in\Gamma'$.
\end{enumerate}
\end{defin}

Next, it is easy to check that whenever $\{\Gamma_n\}_{n\in\en}$ is a sequence of unbounded and $\sigma$-closed sets in $\Gamma$, then $\bigcap_{n\in\en}\Gamma_n$ is again unbounded and $\sigma$-closed in $\Gamma$.

Let us fix a subspace $Y$ of $\C(K)$. In order to see that there is a ``simultaneous projectional skeleton for $\C(K)$ and $Y$'', it is enough to find an unbounded and $\sigma$-closed set $\Gamma'\subset\Gamma$ such that $P_s(Y)\subset Y$ for every $s\in\Gamma'$. Then obviously $\{P_s\!\!\upharpoonright_Y\}_{s\in\Gamma'}$ is 1-projectional skeleton in $Y$ and $\{P_s\}_{s\in\Gamma'}$ is 1-projectional skeleton in $\C(K)$.

If we were able to find such an unbounded and $\sigma$-closed set $\Gamma'\subset\Gamma$ for every $\tau_p(D)$-closed subspace of $\C(K)$, then Theorem \ref{tMain} would easily follow using the fact that we may intersect countably many unbounded and $\sigma$-closed sets as mentioned above.

This is done in the following proposition. The proof is quite technical and its idea comes from \cite{valdivia}, where a similar statement concerning PRI is proved. In the proof we do not need $Y$ to be a subspace, so we formulate it in a more general way.

\begin{proposition}\label{pClosedSubspace}Let $K$ be a compact space with a retractional skeleton $\mathfrak{s} = \{r_s\}_{s\in\Gamma}$ and put $D = D(\mathfrak{s})$. Let $Y$ be a $\tau_p(D)$-closed subset of $\C(K)$. Then there exists an unbounded and $\sigma$-closed set $\Gamma'\subset\Gamma$ such that, for every $t\in\Gamma'$ and $f\in Y$ we have $f\circ r_t\in Y$.\end{proposition}
\begin{proof}In the proof we denote by $\O$ the set of all the rational open intervals in $\er$. If $K_1,\ldots,K_n$ are subsets of $K$ and $o_1,\ldots,o_n\in\O$, we put
$$T(K_1,K_2,\ldots,K_n;o_1,o_2,\ldots,o_n) = \{f\in\C(K):\;f(K_i)\subset o_i\text{ for any }i=1,2,\ldots,n\}.$$
Let us define, for every $s\in\Gamma$, the projection $P_s:\C(K)\to\C(K)$ by $P_s(f) = f\circ r_s$, $f\in\C(K)$. By {\cite[Proposition
28]{kubis}}, $\{P_s\}_{s\in\Gamma}$ is a 1-projectional skeleton in $\C(K)$. Put 
$$\Gamma' = \{s\in\Gamma:\;P_s(Y)\subset Y\}.$$
Using Remark \ref{remarkSkeleton}, it is easy to verify that $\Gamma'$ is $\sigma$-closed set. In order to show that it is unbounded, let us fix some $s\in\Gamma$ and put $s_1 = s$. We inductively define increasing sequences $(s_n)_{n\in\en}$ in $\Gamma$ and $(\U_n)_{n\in\en}$ in the following way.

Whenever $s_n\in\Gamma$ is given, let $\U_n$ be a countable basis of the topology on $r_{s_n}[K]$. For all $k\in\en$,
$U_{j_1},\ldots,U_{j_k}\in\U_n$, $o_{m_1},\ldots,o_{m_k}\in\O$ we fix, if it exists, a set $\{x_1,\ldots,x_k\}\subset D$ such that
$$T(r_{s_n}^{-1}(U_{j_1}),r_{s_n}^{-1}(U_{j_2}),\ldots,r_{s_n}^{-1}(U_{j_k});o_{m_1},\ldots,o_{m_k})\;\subset\; T(\{x_1\},\{x_2\},\ldots,\{x_k\};o_{m_1},\ldots,o_{m_k}),$$
and that the latter set is a subset of $\C(K)\setminus Y$.\\
Now, we find $s_{n+1} > s_n$ such that $r_{s_{n+1}}[K]$ contains all the points $\{x_1,\ldots,x_k\}$ corresponding to all
$$k\in\en,\quad U_{j_1},\ldots,U_{j_k}\in\U_n,\quad o_{m_1},\ldots,o_{m_k}\in\O.$$

We define $t = \sup s_n$. Now, it remains to show that $P_t(Y)\subset Y$. Arguing by contradiction, let us assume that there exists an $f\in Y$ such that $P_t(f)\notin Y$. Then there are $k\in\en$, $z_1,\ldots,z_k\in D$ and $o_1,\ldots,o_k\in \O$ with
$$P_t(f)\in T(\{z_1\},\{z_2\},\ldots,\{z_k\};o_1,o_2,\ldots,o_k)\subset \C(K)\setminus Y.$$
Now, fix $\eps > 0$ such that, for every $i\in\{1,\ldots,k\}$,
$$[P_t(f)(z_i)-3\eps,P_t(f)(z_i)+3\eps]\subset o_i.$$
Using the fact that $\{P_s\}_{s\in\Gamma}$ is a 1-projectional skeleton in $\C(K)$ and Remark \ref{remarkSkeleton}, we find $n\in\en$ with $\|P_{s_n}(f) - P_t(f)\| < \eps$. By the continuity of $P_{s_n}(f)\!\!\upharpoonright_{r_{s_n}[K]} = f\!\!\upharpoonright_{r_{s_n}[K]}\in\C(r_{s_n}[K])$, for every $i\in\{1,\ldots,k\}$, there is $U_i\in\U_n$ with $r_{s_n}(z_i)\in U_i$ and 
$$f(U_i)\subset(P_{s_n}(f)(z_i)-\eps,P_{s_n}(f)(z_i)+\eps).$$
Thus, for every $x\in r_{s_n}^{-1}(U_i)$,
\begin{align*}|P_t(f)(x) - P_t(f)(z_i)| \leq & |P_t(f)(x) - P_{s_n}(f)(x)| + |P_{s_n}(f)(x) - P_{s_n}(f)(z_i)|  + \\ & |P_{s_n}(f)(z_i) - P_t(f)(z_i)| \leq 3\eps.\end{align*}
Hence,
\begin{align*}P_t(f)\in T(r_{s_n}^{-1}(U_1),\ldots,r_{s_n}^{-1}(U_k);o_1,\ldots,o_k) & \subset T(\{z_1\},\ldots,\{z_k\};o_1,\ldots,o_k)\\ & \subset \C(K)\setminus Y.\end{align*}
By the construction of the sequence $s_n$, there exists $\{x_1,\ldots,x_k\}\subset r_{s_{n+1}}[K]$ such that
\begin{align*}T(r_{s_n}^{-1}(U_1),r_{s_n}^{-1}(U_2),\ldots,r_{s_n}^{-1}(U_k);o_1,\ldots,o_k) & \subset 
T(\{x_1\},\{x_2\},\ldots,\{x_k\};o_1,\ldots,o_k)\\ & \subset \C(K)\setminus Y.\end{align*}
Consequently,
$$f(x_i) = P_{s_{n+1}}(f)(x_i) = P_t(f)(x_i)\in o_i,\quad i = 1,\ldots,k$$
and
$$f\in T(\{x_1\},\{x_2\},\ldots,\{x_k\};o_1,\ldots,o_k) \subset \C(K)\setminus Y,$$
which is a contradiction with $f\in Y$.
\end{proof}

Let recall that Theorem \ref{tMain} easily follows from Proposition \ref{pClosedSubspace}, as mentioned above. Moreover, we easily obtain the following more precise and more technical statement.

\begin{cor}\label{cSimultaneousSkeleton}Assume $D$ is induced by a retractional skeleton $\{r_s\}_{s\in\Gamma}$ in $K$. Let
$\{P_s\}_{s\in\Gamma}$ be the $1$-projectional skeleton in $\C(K)$ induced by $\{r_s\}_{s\in\Gamma}$; i.e., $P_s(f) = f\circ r_s$,
$s\in\Gamma$, $f\in\C(K)$.

Let $(F_n)_{n=1}^\infty$ be a sequence of closed subset in $K$ such that $F_n\cap D$ is dense in $F_n$ for all $n\in\en$. Let $(Y_n)_{n=1}^\infty$ be a sequence of $\tau_p(D)$-closed subsets of $\C(K)$. Then there is an up-directed, unbounded and $\sigma$-closed set $\Gamma'\subset \Gamma$ such that, for all $n\in\en$, $r_s[F_n]\subset F_n$ and $P_s[Y_n]\subset Y_n$.

In particular, for every $n\in\en$, $\{r_s\!\!\upharpoonright_{F_n}\}_{s\in\Gamma'}$ is a retractional skeleton in $F_n$ and $\{P_s\!\!\upharpoonright_{Y_n}\}_{s\in\Gamma'}$ is a $1$-projectional skeleton in $Y_n$ if $Y_n$ is a subspace.
\end{cor}
\begin{proof}Recall, that the intersection of countably many unbounded and $\sigma$-closed sets in $\Gamma$ is again an unbounded and $\sigma$-closed set in $\Gamma$. Thus, it is enough to use Proposition \ref{pClosedSubspace} and the proof of {\cite[Lemma 3.5]{cuthCMUC}} to construct a sequence of unbounded and $\sigma$-closed sets $\{\Gamma_n\}_{n\in\en}$ such that, for every $s\in\Gamma_n$, $r_s[F_n]\subset F_n$ and $P_s[Y_n]\subset Y_n$.
\end{proof}

\section{Consequences of the existence of a simultaneous projectional skeleton}

We use the existence of a ``simultaneous projectional skeleton'' to obtain certain new results concerning the structure of
spaces with a projectional (resp. retractional) skeleton. Those are similar results to the ones from \cite{kalendaSurvey}, concerning spaces with a commutative projectional (resp. retractional) skeleton; i.e., Plichko spaces and Valdivia compacta. Let us remark that Theorem \ref{tAnswerCuth} gives an answer to {\cite[Question 1]{cuthCMUC}}.

The following two theorems give the relationship between 1-projectional and retractional skeletons.

\begin{thm}\label{tAnswerCuth}Let $K$ be a compact space. Then the following conditions are equivalent:
\begin{enumerate}[\upshape (i)]
	\item $\C(K)$ has a $1$-projectional skeleton.
	\item There is a convex symmetric set induced by a retractional skeleton in $(B_{\C(K)^*},w^*)$.
	\item There is a convex set induced by a retractional skeleton in $(B_{\C(K)^*},w^*)$.
	\item There is a convex set induced by a retractional skeleton in $P(K)$.
\end{enumerate}\end{thm}

\begin{thm}\label{tPSkeletonIffRskeleton}Let $(X,\|\cdot\|)$ be a Banach space. Then the following conditions are equivalent:
\begin{itemize}
	\item[(i)] $X$ has a $1$-projectional skeleton.
	\item[(ii)] There is a convex symmetric set induced by a retractional skeleton in $(B_{X^*},w^*)$.
\end{itemize}
Moreover, if $D$ is a 1-norming subspace of $X^*$, then:
\begin{itemize}
	\item[(iii)] If $D$ is a set induced by a $1$-projectional skeleton in $X$, then $D\cap B_{X^*}$ is induced by a retractional skeleton in $(B_{X^*},w^*)$.
	\item[(iv)] $D$ is a subset of a set induced by a $1$-projectional skeleton in $X$ if and only if $D\cap B_{X^*}$ is a subset of a set induced by a retractional skeleton in $(B_{X^*},w^*)$.
\end{itemize}
\end{thm}

Let us note that by \cite{kalenda02} there is a Banach space which has no PRI (and hence no \mbox{1-projectional} skeleton) but whose dual unit ball is Valdivia (and hence it has a retractional skeleton). Thus, Theorem \ref{tPSkeletonIffRskeleton} does not hold without the assumption on convexity and symmetry in (ii). However, the answer to the following question seems to be unknown.

\begin{question}Let $(X,\|\cdot\|)$ be a Banach space such that there is a convex set induced by a retractional skeleton in $(B_{X^*},w^*)$. Does $X$ have a $1$-projectional skeleton?
\end{question}

\begin{proof}[Proof of Theorem \ref{tAnswerCuth}]Implication (i)$\Rightarrow$(ii) is proved in {\cite[Proposition 3.15]{cuthCMUC}}, (ii)$\Rightarrow$(iii) is obvious and (iii)$\Rightarrow$(iv) follows from Lemma \ref{lBasicSkeleton}. Thus, it remains to prove (iv)$\Rightarrow$(i). Let us fix a convex set $D$ induced by a retractional skeleton in $P(K)$. Let us consider the injection $I:\C(K)\to \C(P(K))$ defined by $I(f)(\mu) = \mu(f)$, $\mu\in P(K)$, $f\in\C(K)$. Notice, that $F\in\C(P(K))$ belongs to $I(\C(K))$ if and only if $F$ is affine.

Indeed, obviously every $f\in I(\C(K))$ is affine. Moreover, if $F\in\C(P(K))$ is affine, we define $f\in\C(K)$ by $f(x) = F(\delta_x)$, where $\delta_x$ is the Dirac measure on $K$ supported by $x\in K$. Then $I(f) = F$.

Moreover, $I(\C(K))$ is a $\tau_p(D)$-closed subset in $\C(P(K))$. Indeed, let $F_{\nu}\stackrel{\tau_p(D)}{\to}F$ where $F_{\nu}\in I(\C(K))$ and $F\in \C(P(K))$. Using the fact that $D$ is convex and $F_{\nu}$ are affine, $F\!\!\upharpoonright_D$ is affine. As $D$ is dense in $P(K)$, $F$ is affine and hence $F\in I(\C(K))$.

By Theorem \ref{tMain}, $I(\C(K))$ has a 1-projectional skeleton. As $I(\C(K))$ is isometric to $\C(K)$, $\C(K)$ has a 1-projectional skeleton as well.
\end{proof}

Let us recall the following well-known lemma. Its proof can be found for example in {\cite[Lemma 2.14]{kalenda00}}.

\begin{lemma}\label{lXIsClosedInCBX}Let $X$ be a Banach space. Consider the isometry $I:X\to\C(B_{X^*},w^*)$ defined by $I(x)(x^*) = x^*(x)$, $x\in X$, $x^*\in B_{X^*}$. Then $f\in \C(B_{X^*},w^*)$ is an element of $I(X)$ if and only if $f$ is affine and $f(0) = 0$.

Moreover, if $D$ is a dense convex symmetric set in $B_{X^*}$, then $I(X)$ is $\tau_p(D)$-closed subset in $\C(B_{X^*},w^*)$.
\end{lemma}

Now we are ready to prove the second theorem.

\begin{proof}[Proof of Theorem \ref{tPSkeletonIffRskeleton}]The implication (i)$\Rightarrow$(ii) and the assertion (iii) are proved in {\cite[Proposition 3.14]{cuthCMUC}}. The ``only if'' part in (iv) follows from (iii).

Let us continue with proving (ii)$\Rightarrow$(i). Fix a convex symmetric set $D$ induced by a retractional skeleton in $(B_{X^*},w^*)$. Consider the isometry $I:X\to\C(B_{X^*},w^*)$ defined by $I(x)(x^*) = x^*(x)$, $x\in X$, $x^*\in B_{X^*}$. By Lemma \ref{lXIsClosedInCBX},  $I(X)$ is a $\tau_p(D)$-closed subset in $\C(B_{X^*},w^*)$. By Theorem \ref{tMain}, $I(X)$ has a 1-projectional skeleton. Thus, $X$ has a \mbox{1-projectional} skeleton and (ii)$\Rightarrow$(i) holds.

It remain to prove the ``if'' part of (iv). Let $D$ be a subspace of $X^*$ and $\mathfrak{s} = \{r_s\}_{s\in\Gamma}$ be a retractional skeleton in $(B_{X^*},w^*)$
with $D\cap B_{X^*}\subset D(\mathfrak{s})$. By Lemma \ref{lXIsClosedInCBX}, $I(X)$ is $\tau_p(D\cap B_{X^*})$-closed in $\C(B_{X^*},w^*)$; hence, it is also $\tau_p(D(\mathfrak{s}))$-closed. By Proposition \ref{pClosedSubspace}, we may without loss of generality assume that $\{P_s\!\!\upharpoonright_{I(X)}\}_{s\in\Gamma}$ is a 1-projectional skeleton in $I(X)$, where $P_s(f) = f\circ r_s$, $s\in\Gamma$, $f\in\C(B_{X^*},w^*)$. Hence, $\mathfrak{s}_X = \{I^{-1}\circ P_s\circ I\}_{s\in\Gamma}$ is a 1-projectional skeleton in $X$. In order to verify that $D\subset D(\mathfrak{s}_X)$, fix $d\in D$ and $s\in\Gamma$ such that $r_s(d) = d$. Fix $x\in X$. Then
\begin{align*}
(I^{-1}\circ P_s\circ I)^*(d)(x) & = (d\circ I^{-1})\; (P_s (I(x))) = (d\circ I^{-1})\;(I(x)\circ r_s)\\ & = (I(x)\circ r_s)(d) = I(x)(r_s(d)) = I(x)(d) = d(x).
\end{align*}
Thus, $(I^{-1}\circ P_s\circ I)^*(d) = d$ and $d\in D(\mathfrak{s}_X)$. Hence, (iv) holds.
\end{proof}

The following two theorems give a finer idea on when continuous image (resp. subspace) of a space with a retractional (resp. projectional) skeleton has again a retractional (resp. projectional) skeleton.

\begin{thm}\label{tContImage}Let $\varphi:K\to L$ be a continuous surjection between compact spaces. Assume $D$ is induced by a retractional skeleton in $K$ and put $B = \varphi(D)$. Then the following conditions are equivalent:
\begin{enumerate}[\upshape (i)]
	\item $B$ is induced by a retractional skeleton in $L$.
	\item $\varphi^*C(L) = \{f\circ\varphi:\;f\in\C(L)\}$ is $\tau_p(D)$-closed in $\C(K)$.
	\item $L = \beta B$ and $B$ is a Fr\'echet-Urysohn space.
	\item $L = \beta B$ and $\varphi\!\!\upharpoonright_D$ is a quotient mapping of $D$ onto $B$.
\end{enumerate}
\end{thm}

\begin{thm}\label{tSubspaceSkeleton}Let $(X,\|\cdot\|)$ be a Banach space, $Y\subset X$ a subspace and $D\subset X^*$ a set induced by $1$-projectional skeleton.  Then the following conditions are equivalent:
\begin{enumerate}[\upshape (i)]
	\item $D\!\!\upharpoonright_{Y}$ is induced by a $1$-projectional skeleton in $Y$.
	\item $D\!\!\upharpoonright_{Y}\cap B_{Y^*}$ is induced by a retractional skeleton in $(B_{Y^*},w^*)$.
	\item $Y$ is $\sigma(X,D)$-closed in $X$.
	\item $\beta((D\cap B_{X^*})\!\!\upharpoonright_{Y},w^*) = (B_{Y^*},w^*)$ and $((D\cap B_{X^*})\!\!\upharpoonright_{Y},w^*)$ is a Fr\'echet-Urysohn space.
	\item $\beta((D\cap B_{X^*})\!\!\upharpoonright_{Y},w^*) = (B_{Y^*},w^*)$ and $R:d\to d\!\!\upharpoonright_{Y}$ is a quotient mapping of $(D\cap B_{X^*},w^*)$ onto its image in  $(B_{Y^*},w^*)$.
\end{enumerate}
\end{thm}
\begin{proof}[Proof of Theorem \ref{tContImage}]The assertion (i)$\Rightarrow$(iii) follows from Lemma \ref{lBasicSkeleton}. Assume (iii) is true. Then, using Lemma \ref{lFrechetQuotient}, $\varphi\!\!\upharpoonright_D$ is closed, and therefore a quotient mapping. Hence, (iii)$\Rightarrow$(iv) is proved.

(iv)$\Rightarrow$(ii)\; Assume that (iv) holds and fix a net of functions $f_\nu$ from $\C(L)$ such that
$f_\nu\circ\varphi\stackrel{\tau_p(D)}{\to} g\in\C(K)$. Now, define function $f$ as $f(\varphi(d)) = g(d)$, $d\in D$. As $\varphi\!\!\upharpoonright_D$ is a quotient mapping, $f$ is continuous and bounded (and defined on $B$). Hence, there is a continuous extension $\tilde{f}\in\C(L)$, $\tilde{f}\supset f$. As $\tilde{f}\circ \varphi = g$ on the dense set $D$, $\tilde{f}\circ \varphi = g$ on $K$ and $g\in \varphi^*C(L)$. Thus, (iv)$\Rightarrow$(ii) is proved.

(ii)$\Rightarrow$(i)\; Assume that (ii) holds. Let $\mathfrak{s} = \{r_s\}_{s\in\Gamma}$ be a retractional skeleton in $K$ such
that $D(\mathfrak{s}) = D$. By Proposition \ref{pClosedSubspace}, we can without loss of generality assume that
$\{P_s\!\!\upharpoonright_{\varphi^*\C(L)}\}_{s\in\Gamma}$ is a 1-projectional skeleton in $\varphi^*\C(L)$, where $P_s(f) = f\circ
r_s$, $f\in\C(K)$, $s\in\Gamma$. In the rest of this proof we will denote by $Y$ (resp. $T_s$) the space $\varphi^*\C(L)$ (resp.
projections $P_s\!\!\upharpoonright_{\varphi^*\C(L)}$). Recall that by \cite{kubis}, $\{T^*_s\!\!\upharpoonright_{B_{Y^*}}\}_{s\in\Gamma}$ is retractional skeleton in $(B_{Y^*},w^*)$; hence, $R = \bigcup_{s\in\Gamma}T^*_s(B_{Y^*})$ is induced by a retractional skeleton in $(B_{Y^*},w^*)$.

Observe, that $L$ is homeomorphic to a subset of $(B_{Y^*},w^*)$. Indeed, let us define the mapping $h:L\to (B_{Y^*},w^*)$ by $h(l) = \delta_{\varphi^{-1}(l)}\!\!\upharpoonright_{Y}$, where $\delta_{\varphi^{-1}(l)}$ is the Dirac measure on $K$ supported by a point from ${\varphi^{-1}(l)}$. It is easy to observe that $h$ is a homeomorphism onto $h(L)$.

Now, we will verify that $h(\varphi(D))\subset h(L)\cap R$. Fix $s\in\Gamma$ and $k\in K$. We would like to see that $\mu = h(\varphi(r_s(k)))\in R$. Hence, we need to see $T^*_s(\mu) = \mu$. Fix $f\in\C(L)$. Then $P_s(f\circ\varphi)\in Y$; hence, there exists $g\in\C(L)$ such that $f\circ\varphi\circ r_s = g\circ \varphi$. Moreover, $f\circ\varphi\circ r_s = f\circ \varphi$ on $r_s[K]$; thus, $g\circ \varphi = f\circ\varphi$ on $r_s[K]$. Now,
$$T^*_s(\mu)(f\circ\varphi) = \mu(f\circ\varphi\circ r_s) = \mu(g\circ\varphi) = (g\circ\varphi)(r_s(k)) = (f\circ\varphi)(r_s(k)) = \mu(f\circ\varphi),$$
and $T^*_s(\mu) = \mu$.

Using the above and the fact that $\varphi(D)$ is dense in $L$, $h(L)\cap R$ is dense in $h(L)$. By Lemma \ref{lBasicSkeleton}, $h(L)\cap R$ is induced by a retractional skeleton in $h(L)$. By Lemma \ref{lBasicSkeleton} (v), $h(\varphi(D)) =  h(L)\cap R$. Hence, $\varphi(D)$ is induced by a retractional skeleton in $L$. This finishes the proof.
\end{proof}

\begin{proof}[Proof of Theorem \ref{tSubspaceSkeleton}]By Theorem \ref{tPSkeletonIffRskeleton} (iii), (i)$\Rightarrow$(ii) is true.

(ii)$\Rightarrow$(iv)\;Let us assume that (ii) holds. Then $D\!\!\upharpoonright_{Y}\cap
B_{Y^*}$ (resp. $D\cap B_{X^*}$) is induced by a retractional skeleton in $(B_{Y^*},w^*)$ (resp. $(B_{X^*},w^*)$). By Lemma
\ref{lBasicSkeleton}, $D\!\!\upharpoonright_{Y}\cap B_{Y^*}$ (resp. $D\cap B_{X^*}$) is dense and countably compact in $(B_{Y^*},w^*)$ (resp. $(B_{X^*},w^*)$). Let us consider the injection $I:Y\hookrightarrow X$. Then $I^*$ is $w^*-w^*$ continuous and $I^*(D\cap B_{X^*}) = (D\cap B_{X^*})\!\!\upharpoonright_{Y}$ is dense and countably compact in $(B_{Y^*},w^*)$. By Lemma \ref{lBasicSkeleton} (v), $(D\cap B_{X^*})\!\!\upharpoonright_{Y} = D\!\!\upharpoonright_{Y}\cap B_{Y^*}$ and (iv) holds.

Assume (iv) is true. Then, using Lemma \ref{lFrechetQuotient}, $R$ is closed, and therefore a quotient mapping. Hence, (iv)$\Rightarrow$(v) is proved.

(v)$\Rightarrow$(iii)\; Assume that (v) holds and fix a net $y_\nu$ from $Y$ such that $y_\nu\stackrel{\tau_p(D)}{\to} x\in X$. Now, define function $y$ as $y(R(d)) = d(x)$, $d\in D$. Since $R$ is a quotient mapping, $y$ is continuous and bounded (and defined on $(D\cap B_{X^*})\!\!\upharpoonright_{Y}$). Hence, there is a continuous extension $\tilde{y}\in \C(B_{Y^*},w^*)$, $\tilde{y}\supset y$. As $\tilde{y}\circ R = x$ on the dense set $D\cap B_{X^*}$, $\tilde{y}\circ R = x$ on $(B_{X^*},w^*)$. Thus, $\tilde{y}$ is affine on $B_{Y^*}$ and $\tilde{y}(0) = 0$. By Lemma \ref{lXIsClosedInCBX}, there exists $z\in Y$ such that $y^*(z) = \tilde{y}(y^*)$ for every $y^*\in B_{Y^*}$. Consequently, $x=z\in Y$ and $Y$ is $\sigma(X,D)$-closed in $X$.

(iii)$\Rightarrow$(i)\; Let $\mathfrak{s} = \{P_s\}_s\in\Gamma$ be the 1-projectional skeleton in $X$ such that $D = D(\mathfrak{s})$ and let $Y$ be $\sigma(X,D)$-closed in $X$. Consider the isometry $I:X\to\C(B_{X^*},w^*)$ defined by $I(x)(x^*) = x^*(x)$, $x\in X$, $x^*\in B_{X^*}$. By Theorem \ref{tPSkeletonIffRskeleton},  $D\cap B_{X^*}$ is induced by a retractional skeleton in $(B_{X^*},w^*)$. By Lemma \ref{lXIsClosedInCBX}, $I(X)$ is a $\tau_p(D\cap B_{X^*})$-closed subset of $\C(B_{X^*},w^*)$. We claim that $I(Y)$ is $\tau_p(D\cap B_{X^*})$-closed in $\C(B_{X^*},w^*)$.

Indeed, let $I(y_{\nu})\stackrel{\tau_p(D\cap B_{X^*})}{\longrightarrow}f$ where $y_{\nu}\in Y$ and $f\in \C(B_{X^*},w^*)$. As
$I(X)$ is $\tau_p(D\cap B_{X^*})$-closed, $f=I(x)$ for some $x\in X$. Now it is easy to observe that
$y_\nu\stackrel{\sigma(X,D)}{\longrightarrow}x$; hence, $x\in Y$. Thus, $f = I(x)\in I(Y)$ and the claim is proved.

Recall that by \cite{kubis}, $\{P^*_s\!\!\upharpoonright_{B_{X^*}}\}_{s\in\Gamma}$ is the retractional skeleton in $(B_{X^*},w^*)$ which induces the set $D\cap B_X^*$ and $\{T_s\}_{s\in\Gamma}$ is a projectional skeleton in $\C(B_{X^*},w^*)$, where $T_s$ is defined by $T_s(f) = f\circ P^*_s\!\!\upharpoonright_{B_{X^*}}$, $s\in\Gamma$, $f\in \C(B_{X^*},w^*)$. By Proposition \ref{pClosedSubspace}, we can without loss of generality assume that $T_s(I(Y))\subset I(Y)$ for every $s\in\Gamma$. Thus, $\mathfrak{s}_Y = \{(I^{-1}\circ T_s\circ I)\!\!\upharpoonright_{Y}\}_{s\in\Gamma}$ is a 1-projectional skeleton in $Y$. It is straightforward to check that, for every $s\in\Gamma$, $(I^{-1}\circ T_s\circ I)\!\!\upharpoonright_X = P_s$. Thus, $\mathfrak{s}_Y = \{P_s\!\!\upharpoonright_{Y}\}_{s\in\Gamma}$ is a 1-projectional skeleton in $Y$ and $D(\mathfrak{s}_Y) = D\!\!\upharpoonright_Y$.
\end{proof}

\section{A new characterization of Asplund spaces}

In \cite{kalenda01} there has been introduced a new class of Banach spaces, $(T)$. A Banach space $X$ belongs to $(T)$ if and
only if $B_X$ is contained in a ``$\Sigma$-subset'' of $(B_{X^{**}},w^*)$; i.e., $B_X$ is contained in a set induced by a
commutative retractional skeleton. Recall that every space from $(T)$ is Asplund. The
class $(T)$ has been used to prove some results
concerning biduals of Asplund spaces. Namely, if the norm on a Banach space $X$ is Kadec, then $X$ is in $(T)$ if and only if
the bidual unit ball is a Valdivia compact space. There has been raised a question, whether $X$ is Asplund whenever the bidual unit ball is Valdivia after every equivalent renorming of $X$. This problem has been solved by O. Kalenda in an unpublished remark, where it is proved that the answer to the problem is positive.

In the following we first observe that, by Theorem \ref{tPSkeletonIffRskeleton}, the noncommutative version of the condition
determining the class $(T)$ gives a characterization of Asplund spaces. In this way, we may look at Asplund spaces as at the ``noncommutative class $(T)$''. Using this observation, we show that ``commutative'' results concerning the class $(T)$
(including the unpublished remark) have
their ``noncommutative'' versions concerning Asplund spaces. In particular, we show that a Banach space $X$ is Asplund if
and only if the bidual unit ball has a retractional skeleton after every equivalent renorming of $X$.

It remains open whether a Banach space $X$ is in $(T)$ whenever the bidual unit ball is Valdivia after every equivalent
renorming of $X$. This question has been already raised in \cite{kalenda01}.

Let us start with the observation that Asplund spaces form exactly the ``noncommutative class $(T)$''.

\begin{thm}\label{tAsplund}Let $(X,\|\cdot\|)$ be a Banach space. Then the following conditions are equivalent:
\begin{enumerate}[\upshape (i)]
	\item $X$ is Asplund.
	\item $X$ is a subset of a set induced by a $1$-projectional skeleton in $X^*$.
	\item $B_X$ is a subset of a set induced by a retractional skeleton in $(B_{X^{**}},w^*)$.
\end{enumerate}
\end{thm}
\begin{proof}The equivalence (i)$\Leftrightarrow$(ii) is proved in {\cite[Proposition 26]{kubis}} (for a simpler proof of (i)$\Rightarrow$(ii) see also \cite{cuthFabian}). The equivalence (ii)$\Leftrightarrow$(iii) follows from Theorem \ref{tPSkeletonIffRskeleton}.
\end{proof}

Notice that, by \cite[Example 4.10]{kalenda01}, $\C(K)^*$ has a commutative 1-projectional skeleton whenever $K$ is a compact space. Thus, condition (ii) in Theorem \ref{tAsplund} cannot be in general replaced by assuming that $X^*$ has a 1-projectional skeleton.

However, if $X$ has a Kadec norm, then the condition (ii) in Theorem \ref{tAsplund} may be weakened in the above mentioned way. This follows from the following ``noncommutative version'' of {\cite[Theorem 4.9]{kalenda01}}. Recall that a norm is called \textit{Kadec} if the norm and weak topologies coincide on the unit sphere, and that each locally uniformly rotund norm is Kadec, see e.g. \cite[Exercise 8.45]{fab}.
\begin{proposition}\label{pKadec}Assume that the norm on a Banach space $(X,\|\cdot\|)$ is Kadec. Then the following assertions are equivalent:
\begin{enumerate}[\upshape (i)]
	\item $X$ is Asplund.
	\item $X^*$ has a $1$-projectional skeleton.
	\item $(B_{X^{**}},w^*)$ has a retractional skeleton.
\end{enumerate}
\end{proposition}
\begin{proof}The implication (i)$\Rightarrow$(ii) follows from Theorem \ref{tAsplund} and (ii)$\Rightarrow$(iii) is a
consequence of Theorem \ref{tPSkeletonIffRskeleton}. Let us assume $D$ is a set induced by a retractional skeleton in
$(B_{X^{**}},w^*)$. Using Theorem \ref{tAsplund}, it is enough to show that $B_X\subset D$. In order to prove it, we follow
the lines of the proof from {\cite[Theorem 4.9]{kalenda01}}, using only Lemma \ref{lBasicSkeleton} instead of {\cite[Lemma 2.4]{kalenda01}}.
\end{proof}

Now we give the new characterization of Asplund spaces we mentioned above.

\begin{thm}\label{tAsplundRenorming}Let $X$ be a Banach space. Then the following assertions are equivalent:
\begin{enumerate}[\upshape (i)]
	\item $X$ is Asplund.
	\item $(X,|\cdot|)^*$ has a $1$-projectional skeleton for every equivalent norm $|\cdot|$ on $X$.
	\item $(B_{(X,|\cdot|)^{**}},w^*)$ has a retractional skeleton for every equivalent norm $|\cdot|$ on $X$.
\end{enumerate}
\end{thm}

Let us recall that in \cite{kalenda01} there is constructed an Asplund space $X$ such that the bidual unit ball does not have a commutative retractional
skeleton; i.e., is not Valdivia. Consequently, $X^*$ does not have a commutative 1-projectional skeleton; i.e., $X^*$ is not 1-Plichko. Thus, conditions (ii) and (iii) in Theorem \ref{tAsplundRenorming} may not be replaced by its
commutative versions. Therefore, the following question, raised already in \cite{kalenda01}, seems to be interesting. It would give a
characterization of those spaces, which have a Valdivia bidual unit ball under every equivalent renorming of $X$.

\begin{question}Suppose that $X$ is a Banach space such that for every equivalent norm on $X$ the bidual unit ball has a
commutative retractional skeleton; i.e., it is Valdivia. Is $X$ in the class $(T)$?
\end{question}

Now we are going to prove Theorem \ref{tAsplundRenorming}. First, we need the following statement. It is an analogy to the statement contained in the unpublished remark by O. Kalenda mentioned above, where the result is proved for the class of Valdivia compact spaces.

\begin{lemma}\label{lAsplund}Let $(X,\|\cdot\|)$ be a Banach space such that $(B_{(X,|\cdot|)^{**},w^*})\in\R_0$ whenever $|\cdot|$ is an equivalent
norm on $X$. Then each subspace of $X$ has the same property.
\end{lemma}
\begin{proof}In order to get a contradiction, let $Y$ be a subspace of $X$ with an equivalent norm $|\cdot|$ such that
$B_{(Y,|\cdot|)^{**}}$ does not have a retractional skeleton. Then $Y$ is a proper subspace of $X$ and hence there are $f\in X^*$ and $x_0\in
X\setminus Y$ such that
$f\!\!\upharpoonright_{Y} = 0$ and $f(x_0) = 1$. The formula
$$\|x\|_1 = |f(x)| + \|x-f(x)x_0\|$$
clearly defines an equivalent norm on $X$.

In the following we will consider any Banach space canonically embedded in its second dual. Further, having a subspace $Z$ of
$X$, we may consider $Z^{**}$ as a subspace of $X^{**}$ (if $i:Z\to X$ is the identity, then $i^{**}$ is a $w^*-w^*$ continuous
linear isometry from $Z^{**}$ onto $(Z^\bot)^\bot = \ov{Z}^{w^*}$; moreover, $X\cap Z^{**} = Z$).

Thus, $M = B_{(Y,|\cdot|)^{**}}$ can be viewed as a $w^*$-compact convex and symmetric subset of $X^{**}$. Put $N = \{F\in
B_{(X,\|\cdot\|_1)^{**}}:\;F(f) = 0\}$ and
$$B = \conv\{N\cup(M+x_0)\cup(M-x_0)\}.$$
Then $B$ is a $w^*$-compact convex and symmetric subset of $X^{**}$. Let us fix a $c>0$ such that $\|y\|\leq c |y|$ for every
$y\in Y$. Then it is easy to verify that
$$\frac 12 B_{(X,\|\cdot\|_1)^{**}}\subset B\subset (1+c)B_{(X,\|\cdot\|_1)^{**}}.$$
Thus, there is an equivalent norm $\|\cdot\|_{**}$ on $X^{**}$ such that $B$ is the unit ball on $(X^{**},\|\cdot\|_{**})$. Moreover, as $B$ is $w^*$-closed,
the norm $\|\cdot\|_{**}$ is a dual norm to some norm $\|\cdot\|_{*}$ on $X^*$ and $B_\circ = \{x^*\in X^*:\;F(x^*)\leq 1\text{
for }F\in B\}$ is the unit ball in
$(X^*,\|\cdot\|_*)$ (see \cite[Fact 5.4]{DGZ}). Notice that $B\cap X$ is $w^*$-dense in $B$.

Indeed, first we put $K = \{F\in
X^{**}:\;F(f) = 0\}$. Now we observe that $K = (\{x\in X:\;f(x) = 0\}^\bot)^\bot$; hence, we may
identify $K$ with $\{x\in X:f(x) = 0\}^{**}$. Then $N\cap X$ may be identified with $K\cap B_X$, which is dense in $N = K\cap
B_{(X,\|\cdot\|_1)^{**}}$. Similarly, $M\cap X$ is dense in $M$. As $N\cap X$ (resp. $M\cap X$) is dense in $N$ (resp. $M$) and
$x_0\in X$, $B\cap X$ is dense in $B$.

Consequently, $B_\circ$ is $w^*$-closed in $X^{*}$ and, by \cite[Fact 5.4]{DGZ}, $\|\cdot\|_{*}$ is a dual norm to some norm on
$X$. Hence, $B$ is bidual unit ball with respect to an equivalent norm on $X$. Now it suffices to observe that $B$ does not have a retractional skeleton.

Let us suppose that $B$ has a retractional skeleton. Then
$$x_0 + M = \{F\in B:\;F(f) = 1\}$$
is a $w^*$-closed $w^*$-$G_\delta$ subset of $B$; hence, by Lemma \ref{lBasicSkeleton}, it has a retractional skeleton. This is
a contradiction with the choice of $M$.
\end{proof}

\begin{proof}[Proof of Theorem \ref{tAsplundRenorming}]The implication (i)$\Rightarrow$(ii) follows from Theorem \ref{tAsplund} and (ii)$\Rightarrow$(iii) is a
consequence of Theorem \ref{tPSkeletonIffRskeleton}. Finally, suppose that $X$ is not Asplund. Then there is a separable subspace
$Y\subset X$ which is not Asplund. Let $|\cdot|$ be an equivalent Kadec norm on $Y$. By Proposition \ref{pKadec},
$B_{(Y,|\cdot|)^{**}}$ does not have a retractional skeleton. Hence, by Lemma \ref{lAsplund}, there is an equivalent norm on
$X$ such that the bidual unit ball does not have a retractional skeleton.
\end{proof}

The following question has already been articulated in \cite{kalenda01} and \cite{avilesKal}.

\begin{question}Let $X$ be an Asplund space. Is there an equivalent norm on $X$ such that $X^*$ has a commutative 1-projectional
skeleton; i.e., is 1-Plichko, or equivalently has a countably 1-norming Markushevich basis?
\end{question}

\section{Some more applications}

In the last section we collect some more applications of the results contained in previous sections. Those are straightforward analogies to results contained in \cite{kalendaSurvey}, where similar statements are proved for Valdivia compact spaces and Plichko spaces.

First, we give some statements concerning open continuous surjections.

\begin{lemma}Let $\varphi:K\to L$ be an open continuous surjection between compact spaces. If $L$ has a dense set of $G_\delta$ points and $D$ is a set induced by a retractional skeleton in $K$, then $\varphi(D)$ is a set induced by a retractional skeleton in $L$.
\end{lemma}
\begin{proof}In order to prove the lemma it is enough to follow the lines of the proof from \cite[Lemma 3.23]{kalendaSurvey}, using only the set induced by a retractional skeleton instead of the dense $\Sigma$-subset. Instead of {\cite[Theorem 3.22]{kalendaSurvey}} and \cite[Lemma 1.11]{kalendaSurvey} we use Theorem \ref{tContImage} and Lemma \ref{lBasicSkeleton}.
\end{proof}

As an immediate consequence we get the following theorem.

\begin{thm}\label{tOpenImage}Let $\varphi:K\to L$ be an open continuous surjection between compact spaces. If $L$ has a dense set of $G_\delta$ points and $K\in\R_0$, then $L\in\R_0$.\end{thm}

It is easy to check that any open continuous image of a compact space with a dense set of $G_\delta$ points has again this property (see \cite[Lemma 4.3]{kalenda00}). Thus, if $K\in\R_0$ has a dense set of $G_\delta$ points and $\varphi$ is an open continuous surjection, then $\varphi(K)\in\R_0$.

However, some assumption on $K$ is needed as there exists a Valdivia compact space $K$ of weight $\aleph_1$ and an open continuous surjection $\varphi$ such that $\varphi(K)$ is not Valdivia (and hence does not have a retractional skeleton); see \cite{kubUsp05} for more details.

Let us have a closer look at products.

\begin{lemma}Let $K$ and $L$ be nonempty compact spaces. If $L$ has a dense set of $G_\delta$ points and $K\times L\in\R_0$, then both $K$ and $L$ have a retractional skeleton as well.
\end{lemma}
\begin{proof}In order to prove the lemma it is enough to follow the lines of the proof from \cite[Proposition 4.7]{kalenda00}, using only Theorem \ref{tOpenImage} and Lemma \ref{lBasicSkeleton} instead of {\cite[Theorem 4.5]{kalenda00}} and \cite[Lemma 1.7]{kalenda00}.
\end{proof}

Let us recall that the class $\R_0$ is closed under arbitrary products (see \cite[Proposition 3.1]{kubis}). Thus, the following theorem follows immediately.

\begin{thm}Let $(K_\alpha)_{\alpha\in A}$ be a collection of nonempty compact spaces such that each $K_\alpha$ has a dense set of $G_\delta$ points. Then $\prod_{\alpha\in A}K_\alpha$ has a retractional skeleton if and only if each $K_\alpha$ has a retractional skeleton.\end{thm}

However, the following question seems to be open.

\begin{question}Suppose that $K$ and $L$ are compact spaces such that $K\times L$ has a retractional skeleton. Do both $K$ and $L$ have a retractional skeleton?
\end{question}

Concerning the stability of the class of spaces with a projectional skeleton, not much is known. Using the results of the previous sections we can obtain some information.

\begin{thm}\label{tSepSubspace}If $X$ is a Banach space with a 1-projectional skeleton and $Y\subset X$ is a separable subspace, then $X/Y$ has a 1-projectional skeleton.\end{thm}
\begin{proof}In the proof we follow the ideas from \cite[Proposition 4.36]{kalendaSurvey}. Let $D$ be a set induced by a 1-projectional skeleton in $X$. Then $D\cap B_{X^*}$ is induced by a retractional skeleton in $B_{X^*}$. As $Y$ is separable, $Y^\bot$ is $w^*$-$G_\delta$ and $w^*$-closed subset of $X^*$. By Lemma \ref{lBasicSkeleton}, $D\cap B_{X^*}\cap Y^\bot$ is a convex symmetric set induced by a retractional skeleton in $B_{X^*}\cap Y^\bot$. Using the identification $(X/Y)^* = Y^\bot$ and Theorem \ref{tPSkeletonIffRskeleton}, $X/Y$ has a 1-projectional skeleton.
\end{proof}
\begin{thm}Let $X$ be a Banach space and $Y\subset X$ a closed subspace such that $X/Y$ is separable. Then:
\begin{enumerate}[\upshape (i)]
	\item If $Y$ is complemented in $X$, then $Y$ has a projectional skeleton if and only if $X$ has a projectional skeleton.
	\item If $Y$ is 1-complemented and $X$ has a 1-projectional skeleton, then $Y$ has a 1-projectional skeleton.
\end{enumerate}
\end{thm}
\begin{proof}Assertion (ii) and the ``if'' part of (i) follow from Theorem \ref{tSepSubspace}. The converse in (i) follows from the fact that the class of spaces with a projectional skeleton is closed under $\ell_1$-sums (see \cite[Theorem 17]{kubis}).
\end{proof}

However, the following question seems to be open.

\begin{question}Does every 1-complemented subspace of a space with a 1-projectional skeleton have a 1-projectional skeleton as well?
\end{question}

\begin{ack}
 The author would like to thank Ond\v{r}ej Kalenda for many useful remarks and discussions, in particular for pointing out the reference \cite{valdivia} and the unpublished remark mentioned above.
\end{ack}

\end{document}